\documentclass[12pt,a4paper]{amsart}
\usepackage[utf8]{inputenc}
\usepackage[english]{babel}
\usepackage[T1]{fontenc}
\usepackage{amsmath}
\usepackage{mathtools}
\usepackage{amsfonts, dsfont}
\usepackage{amssymb}
\usepackage{vmargin}
\usepackage{setspace}
\usepackage{mathrsfs, enumerate, csquotes, color}
\usepackage{xcolor}
\definecolor{vertfonce}{rgb}{0.20, 0.46, 0.25}
\definecolor{rougefonce}{rgb}{0.64, 0.09, 0.20}
\usepackage[breaklinks=true,
colorlinks=true,
linkcolor=rougefonce,
citecolor=vertfonce]{hyperref}
\usepackage{pgf,tikz,pgfplots}
\pgfplotsset{compat=1.15}
\usetikzlibrary{arrows}
\definecolor{zzttqq}{rgb}{0.6,0.2,0.}
\definecolor{xdxdff}{rgb}{0.49019607843137253,0.49019607843137253,1.}
\definecolor{uuuuuu}{rgb}{0.26666666666666666,0.26666666666666666,0.26666666666666666}
\definecolor{bblue}{rgb}{0,0.2,0.7}

\author[V.  Bonnaillie-Noël]{Virginie Bonnaillie-Noël}
\address[ V.  Bonnaillie-Noël]{D\'epartement de math\'ematiques et applications,  \'Ecole normale sup\'erieure, CNRS, Universit\'e PSL, 75005, Paris, France}
\email{Virginie.Bonnaillie@ens.fr}

\author[S. Fournais]{S{\o}ren Fournais}
\address[S. Fournais]{Department of Mathematics,  Aarhus University,  Ny Munkegade 118,  8000 Aarhus~C, Denmark}
\email{fournais@math.au.dk}

\author[A. Kachmar]{Ayman Kachmar}
\address[A. Kachmar]{Lebanese University, Department of Mathematics, Nabatiye, Lebanon}
\address{Center for Advanced Mathematical Sciences (CAMS, American University of Beirut)}
\email{akachmar@ul.edu.lb}

\author[N. Raymond]{Nicolas Raymond}
\address[N. Raymond]{Univ Angers, CNRS, LAREMA, SFR MATHSTIC, F-49000 Angers, France}
\email{nicolas.raymond@univ-angers.fr}

\title[Discrete spectrum of the magnetic Laplacian]{Discrete spectrum of the magnetic Laplacian\\ on perturbed half-planes}

\makeatletter

\@addtoreset{equation}{section}
\makeatother

\usepackage{tikz}
\usetikzlibrary{shapes.misc}
\tikzset{cross/.style={cross out}, minimum size=1pt, draw=black, inner sep =0pt, outer sep=0pt, cross/.default={1pt}}

\theoremstyle{plain}
\newtheorem{theorem}{Theorem}[section]
\newtheorem{lemma}[theorem]{Lemma}

\newtheorem{proposition}[theorem]{Proposition}

\theoremstyle{definition}
\newtheorem{remark}[theorem]{Remark}

\newcommand{\R}{\mathbf{R}}

\newcommand{\Bk}{\color{black}}

\renewcommand{\leq}{\leqslant}	\renewcommand{\geq}{\geqslant}

\newcommand{\dd}{\mathrm{d}}

\newcommand\DG{f_{\star}}

\begin{document}
	
	\begin{abstract}
	The existence  of bound states for the magnetic Laplacian in unbounded domains can be quite challenging in the case of a homogeneous magnetic field.  We provide an affirmative answer for almost flat corners and slightly curved half-planes when the total curvature of the boundary is positive.
	\end{abstract}
	
	\maketitle
	
\section{Introduction}

We consider the Neumann magnetic Laplacian with constant magnetic field $B=1$ on an open set $\Omega_\delta$, which is defined through the following (closed) quadratic form
\[\forall\psi\in H^1_{\mathbf{A}}(\Omega_\delta)\,,\qquad Q_\delta(\psi)=\int_{\Omega_\delta}|(-i\nabla+\mathbf{A})\psi|^2\dd\mathbf{x}\,.\]
Here $\mathbf{A}=(-x_2,0)$ is a vector potential associated with $B=1$ in the sense that
\[\partial_1A_2-\partial_2A_1=1\,.\]
Then, we consider the magnetic Laplacian as the self-adjoint operator $\mathscr{L}_\delta$ associated with $Q_\delta$. We denote by $\lambda(\delta)$ the bottom of its spectrum
\[\lambda(\delta)=\inf_{\substack{\psi\in H^1_{\mathbf A}(\Omega_\delta)\\\psi\not\equiv0}}\frac{\|(-i\nabla+\mathbf A)\psi\|_{L^2(\Omega_\delta)}^2}{\|\psi\|_{L^2(\Omega_\delta)}^2}\,.\]

The study of the spectrum of such operators in different geometries has been the focus of much interest in recent decades. Below, we will give a short overview of results relevant for the present article. For more in depth coverage we refer to \cite{FH10, Raymond}.

In general, when $\Omega_\delta$ is unbounded, we do not know if $\lambda(\delta)$ is an eigenvalue. For instance, in the case when $\Omega_\delta=\Omega_0=:\R\times\R_+$, it is well-known that the spectrum is absolutely continuous and given by the half-line $[\Theta_0,+\infty)$, where $\Theta_0 \approx 0.59$ is a positive universal constant (see \eqref{eq:dG} below). In the present article $\Omega_\delta$ will be a perturbation of the half-plane $\Omega_0$. Let us describe the two types of perturbations that we consider. The first type is a singular perturbation when $\Omega_\delta$ is of the  (corner) form
\begin{equation}\label{eq.Cdelta}
\Omega_\delta=\mathcal{C}_\delta:=\{\mathbf{x}=(x_1,x_2)\in\R\times\R_+ : \frac{x_2}{\tan\delta}>-x_1 \}\,,\quad \delta\in(0,\pi)\,.
\end{equation}
The second type of perturbation under consideration is regular. Consider a bounded continuous compactly supported function $\kappa : \R\to\R$ and, for all $\delta\geq0$, a simple $\mathscr{C}^2$ curve $\gamma_\delta :\R\to\R^2$, parametrized by arc-length and such that its algebraic curvature $\kappa_\delta$ satisfies
\[\kappa_\delta(s)=\delta\kappa(s)\,.\]
Note that $\gamma_\delta''(s)=\delta\kappa(s) \mathbf{n}_\delta(s)$ where $\mathbf{n}_\delta(s)$ is the unit normal such that $\det(\gamma'_\delta,\mathbf{n}_\delta)=1$. We also assume that $\gamma_0(s)=(s,0)$,  $\gamma_\delta(0)=(0,0)$ and that $\delta\mapsto\gamma'_\delta$ is continuous for the uniform topology at $\delta=0$\Bk. Let us now define our perturbed half-space. We write $\R^2\setminus\gamma_\delta=\Gamma_{\delta}^+\sqcup\Gamma_{\delta}^-$ in such a way that $\mathbf{n}_\delta$ is the inward pointing normal to $\partial\Gamma_{\delta}^+$. We let $\Omega_\delta=\Gamma_{\delta}^+$. A  typical example of such a configuration is given by $\gamma_\delta(s)=(s,f_\delta(s))$ where $f_\delta(s)=\delta\int_0^s  g(\sigma)(1-\delta^2g(\sigma)^2)^{-1/2}\dd \sigma$ and $g(\sigma)=\int_0^\sigma \kappa(x)\dd x$,  with $\delta$  chosen  in $(0,\delta_0)$ for a sufficiently small $\delta_0$  to ensure that $\delta g(\sigma)\in  (-1,1)$; in which case $\Gamma_\delta^+=\{(x,y)\in\R^2~|~y> f_\delta(x)\}$.

By the Persson's theorem (see for instance \cite{BN}), it is well-known in both cases that the essential spectrum of this operator $\mathscr{L}_\delta$ is the same as the spectrum of $\mathscr{L}_0$, \emph{i.e.}, $[\Theta_0,+\infty)$. Let us now state our main two results, both establishing the existence of a bound state.

\begin{theorem}[Almost flat corner]\label{thm.main}
We assume that $\Omega_\delta=\mathcal{C}_\delta$.
There exists $\delta_0\in(0,\pi)$ such that, for all $\delta\in(0,\delta_0)$,
\[\lambda(\delta)\leq\Theta_0- \frac{C_1^2}{4} \delta^2+o(\delta^2)<\Theta_0\,,\]
where $C_1>0$ is a  universal constant (defined below in \eqref{eq.C1}).
In particular, the bottom of the spectrum of $\mathscr{L}_\delta$ belongs to the discrete spectrum.
\end{theorem}	

In the corner setting, $\lambda(\delta)$ is only  known to be an eigenvalue for the non-flat situation when $\delta\in(\frac\pi2-\varepsilon,\pi)$, for a sufficiently small $\varepsilon>0$, see \cite{BN, Eetal}.  The regime $\delta\to 0$ considered in the present article tackles the subtle situation when we know that the discrete spectrum tends to disappear. Let us also underline that it is still an open question to know if $\delta\mapsto\lambda(\delta)$ is monotone non-increasing (as suggested by the numerical simulations in \cite{BN}). If one were able to establish this monotonicity, Theorem \ref{thm.main} would imply that $\lambda(\delta)<\Theta_0$ for all $\delta\in(0,\pi)$, and thus the existence of a bound state for all (non flat) convex corners.  

We next state the result for the slightly curved half-plane.

\begin{theorem}[Slightly curved half-plane]\label{thm.main*}
We assume $\Omega_\delta=\Gamma^+_{\delta}$ and $\int_{\R}\kappa(s)\dd s>0$.
There exists $\delta_0>0$ such that, for all $\delta\in(0,\delta_0)$, the discrete spectrum of $\mathscr{L}_\delta$ is non-empty. Moreover, 
\[\lambda(\delta)\leq\Theta_0-\left(\frac{C_1}{2}\int_\R\kappa(s)\mathrm{d}s  \right)^2\delta^2+o(\delta^2)\,,\]
where $C_1$ is the same universal constant as in Theorem \ref{thm.main}.
\end{theorem}	

Our analysis suggests that there is only one simple eigenvalue below $\Theta_0-C\delta^{\frac12}$, for some constant $C>0$.  This would follow from a dimensional reduction as in \cite{FH06,  HK17} or in the Grushin spirit (see \cite{BHR22}). The question of estimating the number of bound states (below $\Theta_0$) remains open. This requires the derivation of a very precise operator near the threshold of the essential spectrum, what is a quite interesting problem sharing similar features as in \cite{BMR14}.

\begin{remark}
Actually, Theorem \ref{thm.main} can be seen as a \emph{formal} consequence of Theorem \ref{thm.main*}.	Indeed, it is possible to exhibit a normal parametrization of $\partial\mathcal{C}_\delta$ by considering
\[\gamma_\delta(s)=(s\mathds{1}_{\R_+}(s)+s\cos\delta\mathds{1}_{\R_-}(s),-s\sin\delta\mathds{1}_{\R_-}(s))\,.\]
In the sense of distributions, we have
\[\gamma''_\delta=(1-\cos\delta,\sin\delta)\mathfrak{d}_0=2\sin\left(\frac\delta2\right)\mathbf{n}_\delta\mathfrak{d}_0\,,\quad \mathbf{n}_\delta=(\sin(\delta/2),\cos(\delta/2))\,,\]
where $\mathfrak{d}_0$ is the Dirac distribution at $0$ and $\mathbf{n}_{\delta}$ is the direction of the bisector of $\mathcal{C}_\delta$. Formally, the curvature is $\kappa_\delta=2\sin\left(\frac\delta2\right)\mathfrak{d}_0$.

\end{remark}

\subsection*{Organization of the article}

In Section \ref{sec.2}, we introduce the constants $\Theta_0$ and $C_1$, related to the de Gennes operator \eqref{eq.deGennes}. Section \ref{sec.3} is devoted to the proof of Theorem \ref{thm.main}, whereas Section \ref{sec.4} deals with that of Theorem \ref{thm.main*}. In both situations, the proof follows by construction of an appropriate trial state. In the corner case, the phase of the trial state (see \eqref{eq.trialcorner}) is reminiscent of the construction appearing in the non-linear setting of \cite{CG}, while the amplitude (see \eqref{eq.w-L}) is obtained by minimizing a new energy functional in \eqref{eq.opt}. In the regular case, we use a tensorized trial state in curvilinear coordinates (see \eqref{eq.trialreg}) involving the bound state of a 1D model operator studied in \cite{Simon} and revisited in Appendix \ref{sec.app}.

\section{The de\,Gennes operator and the constant $C_1$}\label{sec.2}
The material of this section is standard and only included for convenience and to fix notation. We refer to \cite{FH10} for more material and reference to earlier works.
The constants $\Theta_0$ and $C_1$ in Theorems~\ref{thm.main} and \ref{thm.main*} are defined starting from a  family of 1D harmonic oscillators on the  half axis.  For all  $\xi\in\R$,  let us  denote by  $\mu(\xi)$ the first eigenvalue   of the operator
\begin{equation}\label{eq.deGennes}
H_\xi=-\frac{d^2}{dt^2}+(t-\xi)^2\quad{\rm in}~L^2(\R_+),
\end{equation}
with Neumann boundary condition at $0$, $u'(0)=0$. 

\noindent We introduce the de\,Gennes constant
\begin{equation}\label{eq:dG}
\Theta_0=\inf_{\xi\in\R}\mu(\xi)\,.\end{equation}
We know that $\frac12<\Theta_0<1$  and  there exists a unique $\xi_0$ such that
\[\Theta_0=\mu(\xi_0)\,.\]
Furthermore, $\xi_0=\sqrt{\Theta_0}$ and $\mu''(\xi_0)>0$.  Let us denote by $\DG$ the positive normalized ground state of $\Theta_0$,  i.e.
\[ H_{\xi_0}\DG=\Theta_0\DG,\quad \DG>0,\quad \DG'(0)=0,\quad \int_{\R_+}|\DG(t)|^2\dd t=1. \]
We introduce the constant $C_1$ as follows
\begin{equation}\label{eq.C1}
C_1=\frac{|\DG(0)|^2}3\,.
\end{equation}
The function $\DG$ belongs to $\mathcal S(\R_+)$ and decays exponentially at infinity.  
It  satisfies the additional property (Feynman-Hellmann)
\begin{equation}\label{eq.FeyHel}
\int_{\R_+}(t-\xi_0)|\DG(t)|^2\dd t=0\,.
\end{equation}
Noticing that
\[\begin{split}
\int_0^{+\infty}t|\DG'(t)|^2\mathrm{d}t&=-\int_0^{+\infty}(t\DG')' \DG\mathrm{d}t=\int_0^{+\infty}(t(-\DG'')\DG-\DG'\DG)\mathrm{d}t\\
&=\int_0^{+\infty}(\Theta_0t-t(\xi_0-t)^2)|\DG(t)|^2\mathrm{d}t+\frac{\DG(0)^2}{2}\,,
\end{split}\]
 we get, using \eqref{eq.C1}, the interesting identity
\begin{equation}\label{eq.magic}
\int_{\R_+}\big( |f'_{\star}(t)|^2+(\xi_0-t)^2|\DG(t)|^2-\Theta_0|\DG (t)|^2\big)t\dd t= \frac{3C_1}2\,.
\end{equation}
Another interesting identity is
\begin{equation}\label{eq.magicbis}
\int_{\R_+}(t-\xi_0)t(t-2\xi_0)|f_{\star}(t)|^2{\mathrm d}t=\frac{C_1}{2} \,,
\end{equation}
which follows by writing
\[ (t-\xi_0)(t-2\xi_0)t=(t-\xi_0)^3-\xi_0^2(t-\xi_0)\,,\]
using \eqref{eq.FeyHel} and the formula (see \cite[Lemma 3.2.7]{FH10}):
\[\int_{\R_+}(t-\xi_0)^3|f_{\star}(t)|^2{\mathrm d}t=\frac{C_1}2\,.\]
For a positive number $\ell$,  let us introduce the function
\begin{equation}\label{eq.f-ell}
f_\ell(t)={ \zeta}\left(\frac{t}{\ell}\right)\DG(t)
\end{equation} 
where ${\zeta}\in  \mathscr{C}_0^\infty(\R)$ satisfies
\begin{equation}\label{eq.eta}
0\leq {\zeta}\leq 1,\quad {\rm supp}\,{\zeta}\subset[-1,1],\quad {\zeta}=1{\rm~on~}\big[-\frac12,\frac12\big].
\end{equation}
Consequently, as $\ell\to+\infty$, we have
\begin{equation}\label{eq.en-f-ell}
\begin{aligned}\int_{\R_+}|f_\ell(t)|^2\dd t&=1+\mathcal O(\ell^{-\infty}),\\
q(f_\ell)&:=\int_{\R_+}\big(|f_\ell'(t)|^2+(t-\xi_0)^2|f_\ell(t)|^2 \big)\dd t=\Theta_0+\mathcal O(\ell^{-\infty})\,,\\
\int_{\R_+}(t-\xi_0)^k|f_\ell(t)|^2\dd t&=\int_{\R_+}(t-\xi_0)^k|f_{\star}(t)|^2\dd t+\mathcal O(\ell^{-\infty})\,,
\end{aligned}
\end{equation}
where $\mathcal O(\ell^{-\infty})$ denotes a quantity equal to $\mathcal O(\ell^{-N})$ for all $N>0$.\\
Since $\DG\in\mathcal S(\R_+)$, we deduce  the following two identities from \eqref{eq.magic}, which will useful below in our proof of Theorem~\ref{thm.main},
\begin{equation}\label{eq.moment}
\begin{aligned}
\int_{\R_+}\big( |f'_\ell(t)|^2+(\xi_0-t)^2|f_\ell (t)|^2-\Theta_0|f_\ell(t)|^2\big)t\dd t&=\frac{3C_1}2+\mathcal O(\ell^{-\infty})\,,\\
\int_{\R_+}(t-\xi_0)(t-2\xi_0)|f_{\ell}(t)|^2t{\mathrm d}t&=\frac{C_1}{2}+\mathcal O(\ell^{-\infty})\,.
\end{aligned}
\end{equation}
\section{Almost flat  sectors}\label{sec.3}

This section is devoted to the proof of Theorem~\ref{thm.main}, so $\Omega_\delta=\mathcal C_\delta$ hereafter.  
The proof is by construction of a quasi-mode having approximately the form (after truncation)
\[ \Psi^{\rm tr}\approx \DG(t)e^{i\Phi(s,t)}\]
 where $s$ denotes the tangential variable  along $\partial\mathcal C_\delta$ and $t$  denotes the transversal variable.  For  instance,  $(s,t)=(x_1,x_2)$  when  $\delta=0$ (in which  case $\mathcal C_\delta$  is the half-plane $\R\times\R_+$).
 
The phase  term $\Phi$  is,  up to symmetry considerations,  a perturbation of $i\xi_0 s$. As already mentionned, the idea of perturbing the phase term in an almost flat sector was first introduced in the non-linear framework of the Ginzburg-Landau functional \cite{CG}.  We use the same construction and add to this by proving that the phase term proposed in \cite{CG} is rather the optimal choice. Interestingly, we determine the best truncation profile by minimizing a non-linear functional, which allows us to capture the $\delta^2$-term of Theorem \ref{thm.main}.

\subsection{Geometric framework}	
We denote by $T^+$ the following trapezoid: 
\[ T^+:=\{\mathbf{x}\in(0,+\infty)\times(0,\ell) : x_2\tan(\delta/2)<x_1\}\,.\]	
Consider the angle bisector
\[D_\delta=\{\mathbf{x}\in\R^2 : x_1=x_2\tan(\delta/2)\}\,.\]
We denote by $S_\delta$ the reflection in the line $D_\delta$, whose matrix is
\[S_\delta=\begin{pmatrix}
-\cos\delta&\sin\delta\\
\sin\delta&\cos\delta
\end{pmatrix}\,.\]	
We denote by $T^{-}$ the reflection of $T^{+}$, i.e.,
\[ T^{-} := S_{\delta} T^{+}. \]

	\begin{figure}[ht!]
	\begin{tikzpicture}[line cap=round,line join=round,>=triangle 45,x=0.6cm,y=0.6cm]
		\clip(-9,-2) rectangle (13,10);
		\draw [line width=2.pt,domain=-12:0.0] plot(\x,{(-0.--0.848*\x)/-2.172});
		\draw [dashed, line width=1.pt, color=bblue, domain=0:17.308] plot(\x,{(-0.--0.9827311245774308*\x)/0.1850392844419221});
		\draw [dashed, line width=1.pt,domain=-11.972:0.0] plot(\x,{(-0.--5.*\x)/-2.});
		\draw [dashed, line width=1.pt,domain=0.0:17.30800000000001] plot(\x,{(-0.--3.930229175769762*\x)/3.6814804937590164});
		\draw [line width=2.pt,domain=0.0:17.30800000000001] plot(\x,{(-0.-0.*\x)/3.});
		\draw [line width=1.pt,color=zzttqq] plot[domain=2.7693689797066257:3.141592653589793,variable=\t]({1.*7.702886391762553*cos(\t r)+0.*7.702886391762553*sin(\t r)},{0.*7.702886391762553*cos(\t r)+1.*7.702886391762553*sin(\t r)});
		\draw [line width=1.pt,color=zzttqq] plot[domain=0.:0.8180662757993642,variable=\t]({1.*3.*cos(\t r)+0.*3.*sin(\t r)},{0.*3.*cos(\t r)+1.*3.*sin(\t r)}) ;
		\draw [dashed, line width=1.pt,domain=-11.972:0] plot(\x,{(-0.-0.*\x)/3.});
		\draw [dashed, line width=1.pt,domain=0:17.308] plot(\x,{(-0.-0.*\x)/3.});
		\draw [line width=2.pt,domain=1.1:17.308] plot(\x,{(--18.-0.*\x)/3.});
		\draw [line width=1.pt,color=zzttqq]  plot[domain=0.8180662757993641:1.3846844898533128,variable=\t]({1.*8.2211462481647*cos(\t r)+0.*8.2211462481647*sin(\t r)},{0.*8.2211462481647*cos(\t r)+1.*8.2211462481647*sin(\t r)}) ;
		\draw [<->,line width=1.pt] (12.,0.) -- (12.,6.);
		\draw [line width=2.pt,domain=-11.972:1.1297451346409] plot(\x,{(-26.943048179496365--1.6331426637797248*\x)/-4.1830022001527905});
		\draw [fill=uuuuuu] (0.,0.) circle (2.0pt);
		\draw[color=zzttqq] (-8,1.7) node {\Large $\delta$};
		\draw[color=zzttqq] (3.9,1.5) node {\Large $\frac{\pi-\delta-\gamma}{2}$};
		\draw [fill=uuuuuu] (5.620253164556961,6.) circle (2.0pt);
		\draw[color=zzttqq] (4.2,8) node {\Large $\frac{\gamma}{2}$};
		\draw[color=black] (11,3) node {\Large$\ell$};
		\draw[color=bblue] (0.5,9.5) node {\Large$D_\delta$};
		\draw[color=black] (7,4) node {\Large$T^+_{\delta,\gamma}$};
		\draw[color=black] (-6,5) node {\Large$T^-_{\delta,\gamma}$};
		\draw[color=black] (2,4) node {\Large$V^+_{\delta,\gamma}$};
		\draw[color=black] (-.7,5) node {\Large$V^-_{\delta,\gamma}$};
		\draw [fill=uuuuuu] (1.1297451346409,6.) circle (2.0pt);
		\draw [fill=uuuuuu] (-3.0532570655118905,7.633142663779725) circle (2.0pt);
	\end{tikzpicture}
\caption{Geometric setting: The figure illustrates the sector cut-off at height $\ell$. Here $T^+=T^+_{\delta,\gamma} \cup V^+_{\delta,\gamma}$ and similarly for $T^{-}$. The symmetry axis $D_\delta$ is drawn in blue.}\label{fig1}
\end{figure}
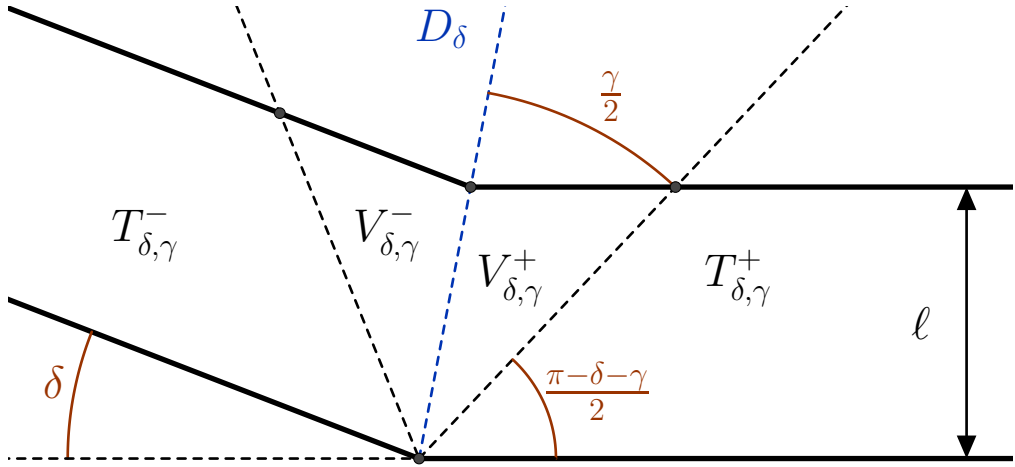
	
\subsection{Towards a test function}\label{sec.sym}
Let us try to define a test function compatible with the symmetry and the magnetic field. Let us  consider a function $u_+$ such that
\begin{equation}\label{eq.E+}
E_+:=\int_{T^+}|(-i\nabla+\mathbf{A})u_+|^2\dd\mathbf{x}<+\infty\,.
\end{equation}
Now, we want to extend $u_+$ by using the symmetry. We wish to do this in such a way that the magnetic energy on $T^-$ coincides with the one on $T^+$. 
\begin{lemma}\label{lem.symm}
Considering
\begin{equation}\label{eq:def_phi}
\phi(\mathbf{y})=\frac{\sin(2\delta)}{4}(y_1^2-y_2^2)-y_1 y_2\sin^2\delta
\end{equation}
and	
\begin{equation}\label{eq.u-}
u_-(\mathbf{x})=e^{-i\phi(S_\delta\mathbf{x})}\overline{u}_+(S_\delta\mathbf{x})\,,
\end{equation}
we have
\begin{equation}\label{eq:equal_energies}
\int_{T^+}|(-i\nabla+\mathbf{A})u_+|^2\dd\mathbf{x}=\int_{T^-}|(-i\nabla+\mathbf{A})u_-|^2\dd\mathbf{x}\,.
\end{equation}
\end{lemma}

\begin{proof}
For a given function $u_{-}$ on $T^{-}$, we
use the change of variable given by $\mathbf{x}=S_\delta\mathbf{y}$ (remember that $S_\delta=S_\delta^{-1}=S_\delta^*$) and notice that
\[\int_{T^-}|(-i\nabla+\mathbf{A})u_-|^2\dd\mathbf{x}=\int_{T^+}|(-i\nabla_{\mathbf{y}}+\tilde{\mathbf{A}}(\mathbf{y})) \overset{\circ}{u}_-|^2\dd\mathbf{y}\,,\]
with
\[\tilde{\mathbf{A}}=S_\delta(\mathbf{A}\circ S_\delta)\,,\quad \overset{\circ}{u}_-=u_-\circ S_\delta(\mathbf{y})\,. \]
A straightforward computation gives
\[\tilde{\mathbf{A}}(\mathbf{y})=S_\delta\begin{pmatrix}
-y_1\sin\delta-y_2\cos\delta\\
0
\end{pmatrix}=\begin{pmatrix}
\frac{y_1}{2}\sin(2\delta)+y_2\cos^2\delta\\
-y_1\sin^2\delta-\frac{y_2}{2}\sin(2\delta)
\end{pmatrix}\,.\]	
The magnetic field associated with $\tilde{\mathbf{A}}$ is	
\[\partial_{y_1}\tilde A_2-\partial_{y_2}\tilde A_1=-1\,.\]	
Let us consider the function $\phi$ defined in \eqref{eq:def_phi}.
It satisfies
\[\tilde{\mathbf{A}}-\begin{pmatrix}
y_2\\
0
\end{pmatrix}=\nabla\phi\,.\]
Therefore,
\[\begin{split}
	\int_{T^-}|(-i\nabla+\mathbf{A})u_-|^2\dd\mathbf{x}&=\int_{T^+}\left|\left(-i\nabla_{\mathbf{y}}+\begin{pmatrix}
	y_2\\
	0
\end{pmatrix}\right)e^{i\phi(\mathbf{y})}\overset{\circ}{u}_-(\mathbf{y})\right|^2\dd\mathbf{y}\\
&=\int_{T^+}\left|\left(-i\nabla_{\mathbf{y}}+\mathbf{A}(\mathbf{y})\right)e^{-i\phi(\mathbf{y})}v_-(\mathbf{y})\right|^2\dd\mathbf{y}\,,\quad \text{ with } v_-=\overline{\overset{\circ}{u}}_-\,,
\end{split}\]
so that \eqref{eq:equal_energies} holds if
\[e^{-i\phi(\mathbf{y})}\overline{\overset{\circ}{u}}_-(\mathbf{y})=u_+(\mathbf{y})\,,\]
or, equivalently,
\[u_-(\mathbf{x})=e^{-i\phi(S_\delta\mathbf{x})}\overline{u}_+(S_\delta\mathbf{x})\,.\]
\end{proof}
\begin{remark}
If we choose
\[u_+(x_1,x_2)=f(x_2) e^{i\xi_0 x_1}\,,\]
with a real-valued $f$ in the Schwartz class, \eqref{eq.u-} gives
\[u_-(\mathbf{x})=f(x_1\sin\delta+x_2\cos\delta)e^{-i\xi_0(-x_1\cos\delta+x_2\sin\delta)}e^{-i\phi(\mathbf{x})}\,,\]
since $\phi(S_\delta\mathbf{x})=\phi(\mathbf{x})$. Of course, this choice of $u_+$ is not appropriate since \eqref{eq.E+} is not satisfied due to the lack of integrability with respect to $x_1$. However, up to using a cutoff function with respect to $x_1$, this gives a rather good idea of the shape of our test function in $T^+$.
Now, if we consider
\[u(\mathbf{x})=\left\{\begin{array}{llc}
	f(x_2)e^{i\xi_0 x_1}\,,&\mbox{ if }\mathbf{x}\in T^+\\
	f(x_1\sin\delta+x_2\cos\delta)e^{i\xi_0(x_1\cos\delta-x_2\sin\delta)}e^{-i\phi(\mathbf{x})}\,,&\mbox{ if }\mathbf{x}\in T^-
\end{array}\right.\,,\]
we see that $u$ does not belong to $H^1$ near the symmetry axis $D_\delta$ due to the phase shift, which does not vanish on $D_\delta$. In the next section, we slightly modify this function $u$ to solve this inconvenience. 
\end{remark}

\subsection{Smoothing the transition near $D_\delta$}
For the given $\delta$, we let $0<\gamma<\pi - \delta$ and define the following trial state (in polar coordinates)
\begin{equation}\label{eq.trialcorner}
\Psi(r\cos\theta,r\sin\theta)=\left\{\begin{array}{llc}
f(r\sin\theta)e^{i\xi_0 r\cos\theta},&\mbox{ if } (r\cos\theta,r\sin\theta)\in T^+_{\delta,\gamma},\\
f(r\sin\theta)e^{i\alpha(r,\theta)},&\mbox{ if }(r\cos\theta,r\sin\theta)\in V^+_{\delta,\gamma},\\
f(r\sin(\theta+\delta))e^{i\alpha(r,\theta)},&\mbox{ if }(r\cos\theta,r\sin\theta)\in V^-_{\delta,\gamma},\\
f(r\sin(\theta+\delta))e^{i\xi_0 r\cos(\theta+\delta)-ir^2\phi(\cos\theta,\sin\theta)},&\mbox{ if }(r\cos\theta,r\sin\theta)\in T^-_{\delta,\gamma},\\
0, & \text{ else},
\end{array}\right.
\end{equation}
where (cf. Fig.\ref{fig1}):
\begin{enumerate}[---]
\item $V^+_{\delta,\gamma}$ is the sector defined in polar coordinates by $0<r<r_*:=\frac{\ell}{\cos\left(\frac{\delta+\gamma}{2}\right)}$ and $\theta\in\left(\theta_\delta-\frac{\gamma}{2},\theta_\delta\right)$ and $V^-_{\delta,\gamma}=S_\delta V^+_{\delta,\gamma}$, with $\theta_\delta=\frac{\pi-\delta}{2}$.

\item the trapezoids $T^+_{\delta,\gamma}$ and $T^-_{\delta,\gamma}$ are given by
\[T^+_{\delta,\gamma}=\left\{\mathbf{x}\in(0,+\infty)\times(0,\ell) : x_2\tan\left(\frac{\delta-\gamma}{2}\right)<x_1\right\}\]
and $T^-_{\delta,\gamma}=S_\delta T^+_{\delta,\gamma}$.	
\end{enumerate}
The function $f$ is given by $f=f_{\ell}$, see \eqref{eq.f-ell}, and the phase $\alpha$ is chosen so that the function $\Psi$ belongs to $H^1_{\mathrm{loc}}(\mathcal{C}_\delta)$. Let us give an explicit choice of $\alpha$. Note that
\[\phi(\cos\theta,\sin\theta)=\frac{\sin\delta}{2}\cos(2\theta+\delta)\,,\]
and consider the two phases
\[\alpha_+(r,\theta)=r\xi_0\cos(\theta)\,,\qquad \alpha_-(r,\theta)=r\xi_0\cos(\theta+\delta)-\frac{r^2}{2}\sin\delta\cos(2\theta+\delta) \,.\]
Notice that the transition zone is simply given by $\theta\in\left(\theta_\delta-\frac{\gamma}{2},\theta_\delta+\frac{\gamma}{2}\right)$ and that we have
\[\alpha_+\left(r,\theta_\delta-\frac{\gamma}{2}\right)=r\xi_0\sin\left(\frac{\delta+\gamma}{2}\right)\]
and
\[\alpha_-\left(r,\theta_\delta+\frac{\gamma}{2}\right)=-r\xi_0\sin\left(\frac{\delta+\gamma}{2}\right)+\frac{r^2}{2}\sin\delta\cos\gamma\,.\]
This leads to the choice
\begin{equation}\label{eq:alpha_def}
	\alpha(r,\theta)=br^2-\chi_{\delta,\gamma}(\theta)\left(ar-br^2\right)\,,
\end{equation}
with 
\[a=\xi_0\sin\left(\frac{\delta+\gamma}{2}\right)\,,\quad b=\frac14\sin\delta\cos\gamma\,,\]
and 
\[\chi_{\delta,\gamma}(\theta)=\chi\left(\frac{2(\theta-\theta_\delta)}{\gamma}\right)\,,\]
where $\chi$ is a  smooth  odd function such that $\chi(t)=1$ for $t\geq 1$.

\begin{remark}
The function $\Psi$ has no decay in the $x_1$-direction, so it will be complemented by a well-chosen cut-off in that variable.
\end{remark}

\subsection{Estimate of the energy}
We consider a function $\eta_+\in H^1(\R_+)$ equal to $1$ in a  fixed  neighborhood of $0$ (so that $(x_1,x_2)\mapsto\eta_+(x_1)$ equals 1 on $V^+_{\delta,\gamma}$ as soon as $\delta$ and $\gamma$ are small enough).  This function will be chosen later on in Section~\ref{sec.eta}.

The aim of this section is to establish the following proposition valid in the regime where $\ell\to+\infty$ and $(\delta,\gamma)\to0$ provided that $\delta=o(\gamma)$.
 
\begin{proposition}\label{prop.ubenergy}
For small enough values of $\delta, \gamma, \ell^{-1}$ and under the assumption that $\delta \leq \gamma$, we have
\begin{multline*}
\int_{T^+}   |\eta_+(x_1)|^2|(-i\nabla+\mathbf{A})\Psi|^2\mathrm{d}\mathbf{x}\leq (\Theta_0+\mathcal{O}(\ell^{-\infty})) \|\eta_+\|_{L^2(\R_+)}^2- J\frac{\delta}{2}\\
+\mathcal O (\gamma^{-1}\delta^2) +\mathcal{O}(\gamma^3)+\mathcal{O}(\ell^{-\infty})\,,
\end{multline*}
	 with 
	 \begin{equation}\label{eq.J}
	 J=\int_0^{+\infty}\left(|f'|^2+(t-\xi_0)^2 |f(t)|^2\right)t\mathrm{d}t
	 -\int_0^{+\infty}(t-\xi_0)(t- 2\xi_0)t|f(t)|^2\mathrm{d}t\,.
	 \end{equation}
	 The remainder terms $\mathcal O (\gamma^{-1}\delta^2)$, $\mathcal O(\ell^{-\infty})$ and $\mathcal O(\gamma^3)$ are controlled uniformly w.r.t. the function $\eta_+$.
\end{proposition}

Proposition \ref{prop.ubenergy} is a consequence of Lemmas \ref{lem.1}, Remark \ref{rem:lem.1}, and \ref{lem.2}.

\begin{lemma}\label{lem.1}
We have the estimate
\[
\int_{V^+_{\delta,\gamma}}|(-i\nabla+\mathbf{A})\Psi|^2\mathrm{d}\mathbf{x}=\frac\gamma2  J_{\delta,\gamma}
+ {\mathcal O}(\gamma^3+ \gamma \ell^{-\infty}) 
\]
where
\[J_{\delta,\gamma}:=
\int_{0}^{\infty}\Big(|f'(t)|^2+\left(t-\left(\xi_0(1+\frac{\delta}{\gamma}) -\frac{\delta}{2\gamma}t\right)\right)^2 |f(t)|^2\Big)t\mathrm{d}t\,.\]
\end{lemma}
\begin{remark}\label{rem:lem.1}
Expanding $J_{\delta,\gamma}$, we observe that
\[ \frac{\gamma}{2} J_{\delta,\gamma}=\frac{\gamma}2\int_0^{+\infty}\left(|f'|^2+(t-\xi_0)^2 |f(t)|^2\right)t\mathrm{d}t+\frac{\delta}{2}\int_0^{+\infty}(t-\xi_0)(t-2\xi_0)t|f(t)|^2\mathrm{d}t+
\mathcal O(\gamma^{-1}\delta^2)\,,\]
where we used the exponential decay of $f_\star$ to control the remainder.
\end{remark}
\begin{proof}[Proof of Lemma~\ref{lem.1}]
We have
\[\int_{V^+_{\delta,\gamma}}|(-i\nabla+\mathbf{A})\Psi|^2\mathrm{d}\mathbf{x}=\int_{V^+_{\delta,\gamma}}|(-i\nabla+\mathbf{A}_0)\tilde\Psi|^2\mathrm{d}\mathbf{x}\,,\quad \mathbf{A}_0=\frac12(-x_2,x_1)\,,\]
where $\tilde\Psi=e^{-ix_1x_2/2}\Psi$. We get
\[\int_{V^+_{\delta,\gamma}}|(-i\nabla+\mathbf{A})\Psi|^2\mathrm{d}\mathbf{x}=\int_{\theta_\delta-\frac{\gamma}{2}}^{\theta_\delta}\int_{0}^{r_*}\left(|\partial_r\tilde\Psi|^2+r^{-2}\left|\left(-i\partial_\theta+\frac{r^2}{2}\right)\tilde\Psi\right|^2\right)r\mathrm{d}r\mathrm{d}\theta\,.\]
Then, we write
\[\tilde{\Psi}=e^{i(-r^2\sin(2\theta)/4+\alpha(r,\theta))}f(r\sin\theta)\,,\]
and, using that $f$ is real-valued, we find
\[|\partial_r\tilde\Psi|^2=\sin^2\theta|f'(r\sin\theta)|^2+(-r\sin(2\theta)/2+\partial_r\alpha)^2|f(r\sin\theta)|^2\,,\]
and
\[\begin{split}
\left|\left(-i\partial_\theta+\frac{r^2}{2}\right)\tilde\Psi\right|^2&=(-\frac{r^2}{2}\cos(2\theta)+\partial_\theta\alpha+\frac{r^2}{2})^2|f(r\sin\theta)|^2+r^2\cos^2\theta|f'(r\sin\theta)|^2\\
&=(r^2\sin^2(\theta)+\partial_\theta\alpha)^2|f(r\sin\theta)|^2+r^2\cos^2\theta|f'(r\sin\theta)|^2\,.
\end{split}\]
It follows that
\begin{align}
&\int_{V^+_{\delta,\gamma}}|(-i\nabla+\mathbf{A})\Psi|^2\mathrm{d}\mathbf{x}\nonumber \\
&=\int_{\theta_\delta-\frac{\gamma}{2}}^{\theta_\delta}\int_{0}^{r_*}\Big(|f'(r\sin\theta)|^2 \nonumber\\
&\qquad+\left((-r\sin(2\theta)/2+\partial_r\alpha)^2+(r\sin^2(\theta)+r^{-1}\partial_\theta\alpha)^2\Big)|f(r\sin\theta)|^2\right)r\mathrm{d}r\mathrm{d}\theta\,\nonumber \\
&=\int_{\theta_\delta-\frac{\gamma}{2}}^{\theta_\delta}\int_{0}^{r_*}\Big(|f'(r\sin\theta)|^2+F(r,\theta,\chi_{\delta,\gamma})|f(r\sin\theta)|^2\Big)r\mathrm{d}r\mathrm{d}\theta\,,
\end{align}
where
\[ F(r,\theta,\chi_{\delta,\gamma})=\left(-r\sin(2\theta)/2-\chi_{\delta,\gamma}(\theta)(a-2br)-2br\right)^2+\left(r\sin^2(\theta)-\chi'_{\delta,\gamma}(\theta)(a-br)\right)^2\,.\]
Notice that the interval of integration in $\theta$ has length $\gamma/2 \ll 1$ and $\theta \in [\theta_\delta-\frac{\gamma}{2},\theta_\delta]$ implies
\[ 1 \geq \sin \theta \geq 1- C (\delta-\gamma)^2, \text{ and } 0 \leq \sin(2\theta) \leq \delta + \gamma.\]
Furthermore,
\[a=\xi_0\frac{\gamma+\delta}{2}+\mathcal O(\gamma^3)\,,\quad b=\frac{\delta}{4}+\mathcal O(\gamma^3)\,\]
and
\[ \chi'_{\delta,\gamma}(\theta) = \frac{2}{\gamma} \chi'\left(\frac{2(\theta-\theta_\delta)}{\gamma}\right). \]
Therefore, the first part of $F(r,\theta,\chi_{\delta,\gamma})$ is small, and we get, using the decay of $f$ and $f'$,
\begin{align}
&\int_{V^+_{\delta,\gamma}}|(-i\nabla+\mathbf{A})\Psi|^2\mathrm{d}\mathbf{x}\nonumber \\
&=
\int_{\theta_\delta-\frac{\gamma}{2}}^{\theta_\delta}\int_{0}^{r_*}\Big(|f'(r\sin\theta)|^2+\left(r\sin^2(\theta)-\chi'_{\delta,\gamma}(\theta)(a-br)\right)^2 |f(r\sin\theta)|^2\Big)r\mathrm{d}r\mathrm{d}\theta + {\mathcal O}(\gamma^3) \nonumber \\
&=
\int_{\theta_\delta-\frac{\gamma}{2}}^{\theta_\delta}\int_{0}^{+\infty}\Big(|f'(t)|^2+\left(t-\chi'\left(\frac{2(\theta-\theta_\delta)}{\gamma}\right)\left(\xi_0(1+\frac{\delta}{\gamma}) -\frac{\delta}{2\gamma}t\right)\right)^2 |f(t)|^2\Big)t\mathrm{d}t\mathrm{d}\theta \nonumber \\
&\quad+ {\mathcal O}(\gamma^3+ \gamma \ell^{-\infty}) \nonumber \\
&=
\frac{\gamma}{2} \int_{0}^{+\infty}\Big(|f'(t)|^2 + t^2 |f(t)|^2\Big)t\mathrm{d}t -
\gamma \int_{0}^{+\infty} \left(\xi_0(1+\frac{\delta}{\gamma}) -\frac{\delta}{2\gamma}t\right) |f(t)|^2 t^2\mathrm{d}t\nonumber \\
&\quad+  \frac{\gamma}{2} \left( \int_{-1}^0 \chi'^2 \mathrm{d}\theta \right) \int_0^{+\infty} \left(\xi_0(1+\frac{\delta}{\gamma}) -\frac{\delta}{2\gamma}t\right)^2 |f(t)|^2 t \mathrm{d}t
 + {\mathcal O}(\gamma^3 + \gamma \ell^{-\infty}), 
\end{align}
where we used that $\chi$ is odd with $\chi(-1)=-1$, to get the last equality.
Therefore, the optimal choice is that $\chi(x) = x$ on $[-1,1]$ and the result follows.
\end{proof}
\begin{lemma}\label{lem.2}
We have
\begin{align*}
\int_{T^+_{\delta,\gamma}}\eta_+(x_1)^2|(-i\nabla+\mathbf{A})\Psi|^2\mathrm{d}\mathbf{x}
&\leq  (\Theta_0+\mathcal{O}(\ell^{-\infty}))\|\eta_+\|_{L^2(\R_+)}^2 \\
&\quad-\frac{\gamma+\delta}{2}\int_0^{+\infty}\big(f'(t)^2+(t-\xi_0)^2|f|^2\big)t\mathrm{d}t\\
&\quad+ {\mathcal O}(\gamma^3) + {\mathcal O}(\ell^{-\infty}).
\end{align*}
\end{lemma}
\begin{proof}
We recall that $\Psi(\mathbf{x})=u(\mathbf{x})=f(x_2)e^{i\xi_0 x_1}$, on $T^+_{\delta,\gamma}$ and that $\eta_+=1$ on $[0,\varepsilon]$. For small enough $\gamma, \delta$ we can therefore write
\[\int_{T^+_{\delta,\gamma}}\eta_+(x_1)^2|(-i\nabla+\mathbf{A})\Psi|^2\mathrm{d}\mathbf{x}=\int_{R}\eta_+(x_1)^2|(-i\nabla+\mathbf{A})u|^2\mathrm{d}\mathbf{x}-E\,,\]
with $R=(0,+\infty)\times(0,\ell)$ and
\[E=\int_{R\setminus T^+_{\delta,\gamma}}|(-i\nabla+\mathbf{A})u|^2\mathrm{d}\mathbf{x}\,.\]
The calculation of $E$ is similar to the beginning of the proof of Lemma~\ref{lem.1}.
We have
\[E=\int_{R\setminus T^+_{\delta,\gamma}}|(-i\nabla+\mathbf{A}_0)\tilde u|^2\mathrm{d}\mathbf{x}\,,\quad \tilde u(\mathbf{x})=e^{-ix_1x_2/2}u(\mathbf{x})\,,\]
with $\mathbf{A}_0(\mathbf{x})=\frac12(-x_2,x_1)$.
We let $w(r,\theta)=\tilde u(r\cos\theta,r\sin\theta)=f(r\sin\theta)e^{i(\xi_0r\cos\theta-\frac{r^2}{4}\sin(2\theta))}$ to get
\[E=\int_{\theta_\delta-\frac\gamma2}^{\frac\pi2}\int_0^{r_*(\theta)}\left(|\partial_rw|^2+r^{-2}\left|\left(-i\partial_\theta+\frac{r^2}{2}\right)w\right|^2\right)r\mathrm{d}r\mathrm{d}\theta\,,\quad r_*(\theta)=\frac{\ell}{\sin\theta}\geq\ell\,.\]
Since $f$ is real-valued,
\[|\partial_rw|^2=\sin^2\theta|f'(r\sin\theta)|^2+(\xi_0\cos\theta-r\sin\theta\cos\theta)^2|f(r\sin\theta)|^2\,\]
\[r^{-2}\left|(-i\partial_\theta+\frac{r^2}{2})w\right|^2=\cos^2\theta|f'(r\sin\theta)|^2+(-\xi_0\sin\theta+r\sin^2(\theta))^2|f(r\sin\theta)|^2\,.\]
We get
\begin{align}\label{eq:Identity}
E= \int_{\theta_\delta-\gamma/2}^{\frac\pi2}\int_0^{\ell/\sin \theta}\left(|f'(r\sin\theta)|^2+(\xi_0-r\sin\theta)^2 |f(r\sin\theta)|^2\right) r\mathrm {d}r\mathrm{d}\theta\,,
\end{align}
Thus, with the change of variable $r=t/\sin\theta$, we get for all $N >0$,
\[E\geq \frac{\gamma+\delta}{2}\int_0^{+\infty}(tf'(t)^2+t(\xi_0-t)^2|f|^2)\mathrm{d}t-C\gamma^3 - C \gamma \ell^{-N}\,.\]
Moreover,
\[\begin{split}\int_{R}\eta_+(x_1)^2|(-i\nabla+\mathbf{A})u|^2\mathrm{d}\mathbf{x}
&=\int_{R}\eta_+(x_1)^2\left( |(-i\partial_1-x_2)u|^2+|\partial_{x_2}u|^2\right)\mathrm{d}\mathbf{x}\\
&=\int_{R}\eta_+(x_1)^2\left( |(\xi_0-x_2)f(x_2)|^2+|f'(x_2)|^2\right)\mathrm{d}\mathbf{x}\\
&=\|\eta_+\|_{L^2(\R_+)}^2\int_{0}^\ell\left( |(\xi_0-x_2)f(x_2)|^2+|f'(x_2)|^2\right)\mathrm{d}x_2\\
&=(\Theta_0+\mathcal{O}(\ell^{-\infty})) \|\eta\|_{L^2(\R_+)}^2\,.
\end{split}\]
\end{proof}

\subsection{Proof of Theorem~\ref{thm.main}}\label{sec.eta}
We have now almost all the elements at hand to establish Theorem \ref{thm.main}.
Let us first truncate the function $\Psi$  to produce a test function in $H^1_{\mathbf A}(\mathcal C_\delta)$.  We introduce  the function
\[\Psi^{\mathrm{tr}}=\eta\Psi\]
where 
\begin{equation}\label{eq.w-L}
\eta(\mathbf x)=\begin{cases}
\eta_+(x_1)&{\rm if~}\mathbf x\in T^+\\
\eta_+(-x_1\cos\delta+x_2\sin\delta)&{\rm  if~}\mathbf x\in T^-\,.
\end{cases}
\end{equation}
and  $\eta_+\in H^1(\R_+)$ is equal to $1$ on the interval $(0,\varepsilon)$, for  a fixed (and arbitrary) $\varepsilon>0$.
The construction of  $\eta$ and $\Psi^{\rm tr}$ respects the symmetry considerations in Section~\ref{sec.sym} (since $\eta\circ S_\delta=\eta$).  Notice that on $T^+$,
\[ \Psi^{\mathrm{tr}}(\mathbf{x})=\eta_+(x_1)\Psi(\mathbf{x})\,.\]

Recall that our main (small) parameter is $\delta \in (0,\pi)$. We have also introduced another small parameter $\gamma$ with $0<\delta < \gamma$ and the large parameter $\ell$ in the definition of $\Psi$. Below we will need the condition that 
\begin{equation}\label{eq:small_ell}
\ell \delta \rightarrow 0.
\end{equation}
We will make the choices
\begin{equation}\label{eq:params}
	\gamma = \delta^{\frac12}\,, \qquad \ell = \delta^{-\frac{1}{2}}\,,
\end{equation}
 the first one ensuring that $\gamma^{-1}\delta^2=\gamma^3$ (see the remainders in Proposition \ref{prop.ubenergy}) and the second one being rather arbitrary. 
The function $\eta_{+}$ will be chosen to depend on $\delta$ at the very end of the proof.

\subsubsection{ Estimates of the $L^2$-norm and of the energy of $\Psi^{\mathrm{tr}}$}
Let us estimate the $L^2$-norm of our test function $\Psi^{\mathrm{tr}}$.  For small enough $\delta$, we have by \eqref{eq.en-f-ell}, and using \eqref{eq:small_ell} in the second term,
\[
\begin{aligned}
\int_{T^+}|\Psi^{\mathrm{tr}}|^2\mathrm{d}\mathbf{x}&=\int_{(0,+\infty)\times(0,\ell)} \eta_+(x_1)^2|f_\ell (x_2)|^2\dd \mathbf x-\int_{\frac{\pi-\delta}{2}}^{\frac\pi2}\int_0^{\frac{\ell}{\sin\theta}}|f(r\sin\theta)|^2r\mathrm{d}r\mathrm{d}\theta\\
&=(1+\mathcal{O}(\ell^{-\infty}))\|\eta_+\|^2_{L^2(\R_+)}-\int_{\frac{\pi-\delta}{2}}^{\frac\pi2}\sin^{-2}\theta\int_0^{\ell}|f(t)|^2t\mathrm{d}t\mathrm{d}\theta\,.
\end{aligned}\]
We deduce that
\begin{equation}\label{eq.L2}
\int_{T^+}|\Psi^{\mathrm{tr}}|^2\mathrm{d}\mathbf{x}=(1+\mathcal{O}(\ell^{-\infty}))\|\eta_+\|^2_{L^2(\R_+)}-\frac\delta2\int_0^{+\infty}t|f(t)|^2\mathrm{d}t+\mathcal{O}(\delta^3)\,.
\end{equation}
Let us now estimate  the energy of $\Psi^{\mathrm{tr}}$.  An integration by parts and a symmetry consideration yield the following identity ($\|\cdot\|$ denotes the norm  in  $L^2(\mathcal C_\delta)$):
\[ \begin{aligned}
Q_\delta(\Psi^{\mathrm{tr}})&=\int_{\mathcal C_\delta} |\eta|^2|(-i\nabla+\mathbf  A)\Psi|^2\dd \mathbf x-\int_{\mathcal C_\delta} \eta\Delta \eta\,|\Psi|^2\dd \mathbf x\\
&=\| \eta(-i\nabla+\mathbf  A)\Psi\|^2-2\int_{T^+} \eta\Delta \eta\,|\Psi|^2\dd \mathbf x\\
&=\| \eta(-i\nabla+\mathbf  A)\Psi\|^2-2\int_{T^+} \eta(x_1) \eta''(x_1)\,|\Psi|^2\dd \mathbf x\\
&=\| \eta(-i\nabla+\mathbf  A)\Psi\|^2+2 \| f \|_{L^2(\R_+)}^2 \int_{\R_+}|\eta'_+(x_1)|^2\dd x_1\,.
\end{aligned} \]
Notice, that for the integration by parts we needed $\eta_+\in H^2(\R_+)$, but a density argument gives the identity for all $\eta_+\in H^1(\R_+)$, with $\eta_+=1$ on the interval $(0,\varepsilon)$.

\subsubsection{Upper bound and proof of Theorem \ref{thm.main}}
Using the symmetry of our construction with respect to $D_\delta$, we get (as in Lemma \ref{lem.symm}):
\[\| \eta(-i\nabla+\mathbf  A)\Psi\|^2-\Theta_0\|\Psi^{\mathrm{tr}}\|^2=2( \| \eta_+(-i\nabla+\mathbf A)\Psi\|^2_{L^2(T^+)}-\Theta_0\|\Psi^{\mathrm{tr}}\|^2_{L^2(T^+)})\,.\]
With \eqref{eq:params}, \eqref{eq.L2} and Proposition \ref{prop.ubenergy},  we get
\begin{multline*}
Q_\delta(\Psi^{\mathrm{tr}})-\Theta_0\|\Psi^{\mathrm{tr}}\|^2 \leq-\delta\left(
 J-\Theta_0\int_0^{+\infty}t|f(t)|^2{\mathrm d}t
\right)
+2\|\eta'_+\|^2_{L^2(\R_+)}\\
+\mathcal{O}(\delta^{\frac32})+\mathcal{O}(\delta^{\infty})\|\eta_+\|^2_{H^1(\R_+)}\,.
\end{multline*}
Recalling \eqref{eq.J} and \eqref{eq.moment}, we get
\[ J-\Theta_0\int_0^{+\infty}t|f(t)|^2{\mathrm d}t=C_1+\mathcal{O}(\delta^{\infty})\,,\]
so that
\[Q_\delta(\Psi^{\mathrm{tr}})-\Theta_0\|\Psi^{\mathrm{tr}}\|^2\leq - C_1\delta+2\|\eta'_+\|^2_{L^2(\R_+)}+{\mathcal{O}(\delta^{\frac32})}+\mathcal{O}(\delta^{\infty})\|\eta_+\|^2_{H^1(\R_+)}\,.\]
Therefore,
\begin{multline*}
\frac{Q_\delta(\Psi^{\mathrm{tr}})}{\|\Psi^{\mathrm{tr}}\|^2}\leq\Theta_0+\frac{1}{\|\Psi^{\mathrm{tr}}\|^2}\left( - C_1\delta+2\|\eta'_+\|^2_{L^2(\R_+)}
+{\mathcal{O}(\delta^{\frac32})}+\mathcal{O}(\delta^{\infty})\|\eta_+\|^2_{H^1(\R_+)}\right)\,,
\end{multline*}
and by \eqref{eq.L2} and \eqref{eq.FeyHel}, we have 
\[\frac{1}{\|\Psi^{\mathrm{tr}}\|^2}=\frac{1}{2\|\eta_+\|^2}\left(1+\frac{\xi_0\delta}{2\|\eta_+\|^2}+\mathcal{O}\left(\frac{\delta^2}{\|\eta_+\|^4}\right)+\mathcal{O}(\delta^3)\right)\,,\]
where we used that $\|\eta_+\|^2\geq \varepsilon$. Then,
\begin{equation}\label{eq.ubfinal}
\frac{Q_\delta(\Psi^{\mathrm{tr}})}{\|\Psi^{\mathrm{tr}}\|^2}\leq\Theta_0+\mathcal{I}_\delta(\eta_+)+\mathcal{O}(\mathcal{R}_\delta(\eta_+))
+\mathcal{O}(\delta^\infty\|\eta_+\|^2_{H^1(\R_+)})\,,
\end{equation}
where
\begin{equation}\label{eq.opt}
\mathcal{I}_{\delta}(\eta_+)=\frac{1}{\|\eta_+\|^2}\left({-\frac{C_1\delta}{2}}+\|\eta_+'\|^2_{L^2(\R_+)}\right)\,,
\end{equation}
and
\[\mathcal{R}_\delta(\eta_{+})=\frac{\delta^{\frac32}}{\|\eta_+\|^2}+\delta\frac{\|\eta'_+\|^2}{\|\eta_+\|^4}+\delta^3\frac{\|\eta'_+\|^2}{\|\eta_+\|^2}\,.\]
The estimate \eqref{eq.ubfinal} leads us to minimize the functional $\mathcal{I}_\delta$ over the $\eta_{+}\in H^1(\R_+)$ such that $\eta_{+}=1$ on $(0,\varepsilon)$, $\varepsilon>0$ being fixed. For our purpose, namely to finish the proof of Theorem~\ref{thm.main}, it suffices to come up with a sufficiently good trial $\eta_+$, so we will be brief.
The Euler-Lagrange equation 
\[-\eta_{+}''=\mathcal{I}_\delta(\eta_{+})\eta_{+}\,,\quad { \rm on }\,~ [\varepsilon,+\infty)\,,\]
leads us to consider test functions $\eta_{+} =\eta_\alpha$ of the form
\[ \eta_\alpha(x)=\begin{cases}
 1& {\rm on ~} (0,\varepsilon)\\
e^{\alpha\varepsilon}e^{-\alpha x}& { \rm } x\geq\varepsilon\,,
\end{cases}\]
where $\alpha>0$. We notice that
\[\mathcal{I}_\delta(\eta_\alpha)=\frac{-\frac{ C_1\delta}{2}+\|\eta'_\alpha\|^2_{L^2(\varepsilon,+\infty)}}{\varepsilon+\|\eta_\alpha\|^2_{L^2(\varepsilon,+\infty)}}=\frac{\alpha^2-d\alpha}{1+2\varepsilon\alpha}\,,\qquad d= C_1\delta\,.\]
This last quantity is minimal for $\alpha_\delta=\frac{1}{2\varepsilon}\left(-1+\sqrt{ 1+2\varepsilon  d }\right)$, and we have
\[\alpha_\delta\underset{\delta\to 0}{=} \frac{C_1\delta}{2}  +{\mathcal O}(\delta^2)\,.\]
Therefore, we choose $\alpha=\frac{C_1\delta}{2}$ and we have
\[\mathcal{I}_\delta(\eta_\alpha)=-\frac{ C_1^2}{4}\delta^2 +{\mathcal O}(\delta^3)\,.\]
We notice that
\[ \| \eta_\alpha \|_{L^2(\R_{+})}^2 = \varepsilon+\frac{1}{C_1\delta}\,, \qquad \| \eta_\alpha' \|_{L^2(\R_{+})}^2 = \frac{C_1 \delta}{4}\,, \]
so that
\[\mathcal{R}_\delta(\eta_\alpha)=\mathcal{O}(\delta^{\frac52})\,.\]
With the upper bound \eqref{eq.ubfinal}, this concludes the proof of Theorem~\ref{thm.main}.

\section{Regular perturbation}\label{sec.4}

The purpose of this section is to prove Theorem~\ref{thm.main*}, so $\Omega_\delta=\Gamma_\delta^+$ hereafter. We construct a test function supported in a tubular neighborhood of $\partial\Omega_\delta$ where we can use the  Frenet coordinates.

 We choose $\ell=\delta ^{-\rho}$ with $\rho\in(0,1)$ and we let $B_\ell=\R\times(0,\ell)$. For $\delta$ small enough the classical tubular coordinates,
 \[\Phi_\delta : B_\ell\ni(s,t)\mapsto \gamma_\delta(s)+t \mathbf{n}_\delta(s)\,,\]
 are well-defined  in the sense that $\Phi_\delta$ is injective and induces a local (and then global) $\mathscr{C}^1$-diffeomorphism (the Jacobian of which being $1-t\delta\kappa(s)$). Indeed, by using $\gamma'_0=(1,0)$, the continuity of $\delta\mapsto\gamma'_\delta$, and the Taylor formula, we can check that there exist $c,\delta_0>0$ such that, for all $\delta\in(0,\delta_0)$ and all $(s_1,s_2)\in\mathbb{R}^2$, 
 \[|\Phi_\delta(s_2,t_2)-\Phi_\delta(s_1,t_1)|\geq c(|s_2-s_1|+|t_2-t_1|)\,.\] 
  \Bk We let $\Omega_{\delta,\ell}=\Phi_\delta(B_\ell)$. We consider a function of the form
 \begin{equation}\label{eq.trialreg}
 \psi(s,t)=\zeta(\ell^{-1}t)\DG(t)g(s)\,,
 \end{equation}
 where $\zeta$ is a cutoff function (see \eqref{eq.eta}) and where $g$ has to be determined and will be chosen realvalued and normalised in $L^2({\mathbb R})$.  There exists a suitable phase $\varphi$ such that (see \cite[Lemma F.1.1]{FH10}) if we let 
 \[\Psi=\big(e^{i\varphi}\psi\big)\circ\Phi^{-1}_\delta\,,\]
  which is supported in $\Omega_{\delta,\ell}$, we have
 \[Q_\delta(\Psi)=\tilde Q_\delta(\psi)\,,\]
 where (with $\kappa=\kappa(s)$)
\begin{multline*}\tilde Q_{\delta }(\psi)=\int_{B_{\ell}} (1-t\delta \kappa)|\partial_{t}\psi|^2\dd s\dd t\\
+\int_{B_{\ell}}(1-t\delta \kappa)^{-1}\Big|\Big(-i\partial_{s}+\xi_{0}-t+\frac{\delta \kappa t^2}{2}\Big)\psi\Big|^2\dd s\dd t\,.\end{multline*}
By the exponential decay of $\DG$, we have
\begin{align}\label{eq.regP.*}
\int_{B_{\ell}} (1-t\delta \kappa)|\partial_{t}\psi|^2\dd s\dd t
&= \int_{B_{\ell}} (1-t\delta \kappa)|g|^2|\ell^{-1}\zeta'(\ell^{-1}t)f_{\star} +\zeta(\ell^{-1}t) f'_{\star}|^2\dd s\dd t\, \nonumber \\
&\leq \|f'_{\star}\|^2\|g\|^2- \delta  \left(\int_{0}^\infty t|f'_{\star}|^2\dd t\right)\int_{\mathbb{R}}\kappa |g|^2\dd s+\mathcal{O}(\ell^{-\infty})\|g\|^2\,.
\end{align}
Also,
\begin{equation}\label{eq.regP.*b}
\int_{B_{\ell}}(1-t\delta \kappa)^{-1}\Big|\Big(-i\partial_{s}+\xi_{0}-t+\frac{\delta \kappa t^2}{2}\Big)\psi\Big|^2\dd s\dd t\leq I+\mathcal{O}(\ell^{-\infty})(\|g\|^2+\|g'\|^2)
\end{equation}
where, using first that $g$ is real-valued and then that $(1-t\delta\kappa t)^{-1}=1+\delta \kappa t+\mathcal O(\delta^2t^2\kappa^2)$,
\begin{align}\label{eq:I}
I&=\int_{{\mathbb R} \times {\mathbb R}_{+}}(1-t\delta \kappa)^{-1}\Big|-ig'f_{\star}+\Big(\xi_{0}-t+\frac{\delta \kappa t^2}{2}\Big)f_{\star}g\Big|^2\dd s\dd t\,\nonumber \\
&=\int_{B_{\ell}}(1-t\delta \kappa)^{-1}\left(|g'|^2f^2_\star+|g|^2\Big(\xi_{0}-t+\frac{\delta \kappa t^2}{2}\Big)^2f^2_\star\right)\dd s\dd t\,.\nonumber\\
	&\leq(1+C\delta )\|g'\|^2+\|(t-\xi_{0})f_{\star}\|^2\|g\|^2 \nonumber \\
	&\quad+\left(\int_{\mathbb{R}}\delta \kappa |g|^2\dd s\right) \int_{0}^{+\infty} ((\xi_{0}-t)t^2+t(t-\xi_{0})^2)f^2_\star\dd t
	+C\int_{\mathbb{R}} \delta ^2\kappa^2|g|^2\dd s\,.
\end{align}
Combining this with \eqref{eq.regP.*} and \eqref{eq.regP.*b}, we deduce that
\begin{multline*}
	\tilde Q_{\delta }(\psi)\leq \Theta_{0}\|g\|^2+(1+C\delta )\|g'\|^2\\
	+\left(\int_{0}^{+\infty} \Big[((\xi_{0}-t)t^2+t(t-\xi_{0})^2)f^2_\star-tf'^2_{\star}\Big]\dd t\right)\left(\int_{\mathbb{R}}\delta \kappa |g|^2\dd s\right)\\
	+C\int_{\mathbb{R}} \delta ^2\kappa^2|g|^2\dd s+\mathcal{O}(\delta^\infty)\|g\|^2_{H^1(\R)}\,.
\end{multline*}
Using the decay of  $\DG$ and  \eqref{eq.FeyHel},  the norm of $\psi$ is given by 
\begin{equation}\label{eq.norm.psi.curved}
\begin{aligned}
\|\psi\|^2&=\int_{B_\ell}(1-\delta t\kappa)|g(s)|^2|\DG(t)^2|^2\zeta(\ell^{-1}t)^2\dd t\dd s\\
&=\|g\|^2-\delta\xi_0\int_{\R_+} \kappa|g(s)|^2\dd s+\mathcal{O}(\delta^{\infty})\|g\|^2\,.
\end{aligned}
\end{equation}
It follows that
\begin{multline}\label{eq.min.max.curv}
	\tilde Q_{\delta }(\psi)-\Theta_{0}\|\psi\|^2\leq
	(1+\tilde C\delta )\|g'\|^2\\
	+\delta\int_{0}^{+\infty}\Big( \big((\xi_{0}-t)t^2+t(t-\xi_{0})^2)f^2_\star-tf'^2_{\star}+t\Theta_{0}f^2_\star\Big)\dd t\left(\int_{\mathbb{R}} \kappa |g|^2\dd s\right)\\
	+\mathcal O(\delta^2)\int_\R \kappa^2\,| g(s)|^2\dd s+\mathcal{O}(\delta^{\infty})\|g\|^2\,.\end{multline}
We have to investigate the sign of
\[
(1+\tilde C\delta )\|g'\|^2+A\int_{\mathbb{R}}\delta \kappa |g|^2\dd s+\mathcal O(\delta^2)\int_\R \kappa^2\,| g(s)|^2\dd s+\mathcal{O}(\delta^{\infty})\|g\|^2\,,
\]
with
\[A=\int_{0}^{+\infty} \Big((\xi_{0}-t)t^2+t(t-\xi_{0})^2)f^2_\star-tf'^2_{\star}+t\Theta_{0}f^2_\star\Big)\dd t\,.\]
By using \eqref{eq.magic} and \eqref{eq.magicbis}, we have
\[A=-\frac{3C_1}{2}+\int_0^{+\infty} t(t-\xi_0)(t-2\xi_0)\DG^2(t)\dd t=-\frac{3C_1}{2}+\frac{C_1}{2}=-C_1\,.\]
We choose $g$ such that $\|g\|^2=1$ and  observe that \eqref{eq.norm.psi.curved} yields for $\delta$ small enough,
\[ 1 +\xi_0\delta\int_{\R_+}\kappa|g(s)|^2\dd  s-\mathcal O(\delta^\infty)\leq \|\psi\|^{-2}\leq 1+\mathcal O(\delta)\,,\] 
where the foregoing  lower bound results from the fact that $\xi_0>0$,  $\delta>0$ and, for every natural number $N\geq 3$,  
\[
 \|\psi\|^{-2}=1+\xi_0\delta G+\xi_0^2\delta^2 G^2\Big[ \underset{=1+\mathcal O(\delta)\geq 0 }{\underbrace{1+\sum_{j=1}^{N-2} (2\xi_0\delta)^jG^j}}\Big] +\mathcal  O(\delta^N),\quad
 G=\int_{\R_+}\kappa |g(s)|^2\dd s\,.
 \]
By  the min-max principle,  we deduce from \eqref{eq.min.max.curv}
\[ \lambda(\delta)\leq \Theta_0+ F_\delta(g)+\mathcal{O}(\delta^{\infty})\,,\]
where
\begin{equation}\label{eq.efftivequadraticform}
F_\delta(g)=(1+\hat C\delta)\|g'\|^2+\delta\int_\R \big(-C_1\kappa+\hat C\delta\kappa^2 \big)|g(s)|^2\dd s\,. 
\end{equation}
Minimizing over $g$ and using the analysis in Appendix~\ref{sec.app},  we get if $\int_\R\kappa(s)\dd s>0$,
\[ \lambda(\delta)\leq \Theta_0-\frac{C_1^2}{4}\left(\int_{\R}\kappa(s)ds\right)^2\delta^2+ \mathcal{O}(\delta^3)\,. \]
More precisely, it is enough to consider the trial function
\begin{equation}\label{eq.choice}
g(s)=\sqrt{\delta|\langle V\rangle|} \exp\left(\frac{\delta\langle V\rangle}{2}|s|\right)\,,~ V(s)=-C_1\kappa(s)\,,~ \langle V\rangle=\int_{\R} V(s)\dd s<0\,. 
\end{equation}
Notice that, by dominated convergence, $\int_{\R}\kappa^2 |g(s)|^2\dd s=\mathcal{O}(\delta)\int_{\R}\kappa^2\dd s$.

\subsection*{Acknowledgments}
Part of this work was carried out while AK  visited the University of Angers  in February 2022.  AK and NR are grateful  to the support of the  F\'ed\'eration de  Recherche  Math\'ematiques des Pays de Loire and of the University of Angers.   AK would like to thank B.  Helffer for enlightening discussions.

\appendix
\section{Weak perturbation limit}\label{sec.app}
In this appendix, we explain where the choice \eqref{eq.choice} is coming from.

Let $V\in\mathscr{C}_0^0(\R)$ such that $\langle V\rangle=\int_\R V(x)\dd x<0$ and $\delta\in\R$. Let us consider the self-adjoint operator $L_\delta$ associated with the following quadratic form
\[q_\delta(\psi)=\|\psi'\|^2+\delta \int_{\R}V(x)|\psi(x)|^2\dd x\,,\quad\forall\psi\in H^1(\R)\,,\]
which is the main term in \eqref{eq.efftivequadraticform}, with $V=-C_1\kappa$.

By the rescaling $x=\delta^{-1}y$, we see that $L_\delta$ is unitarily equivalent to $\delta^2M_\delta$ where
\[M_\delta=-\partial^2_y+\delta^{-1}V(\delta^{-1}y)\,.\]
At a formal level, we see that $M_\delta$ converges to $M^{\mathrm{eff}}:=-\partial^2_y+\langle V\rangle\delta_0$, whose spectrum is
\[\left\{-\frac{\langle V\rangle^2}{4}\right\}\cup[0,+\infty)\,,\]
and with grounstate $\exp(\frac{\langle V\rangle}{2} y)$.

Let us make this heuristics rigorous by comparing the quadratic forms. The quadratic form associated with $M_\delta$ is
\[p_\delta(\psi)=\|\psi'\|^2+\int_{\R}\delta^{-1}V(\delta^{-1}y)|\psi|^2\dd y\,,\quad\forall\psi\in H^1(\R)\,.\]	
We denote by $p^{\mathrm{eff}}$ the quadratic form associated with $M^{\mathrm{eff}}$:	
\[p^{\mathrm{eff}}(\psi)=\|\psi'\|^2+\langle V\rangle|\psi(0)|^2\,,\quad\forall\psi\in H^1(\R)\,.\]	
Let us estimate the difference $p_\delta-p^{\mathrm{eff}}$. For all $\psi\in H^1(\R)$, we have
\[\begin{split}
|p_\delta(\psi)-p^{\mathrm{eff}}(\psi)|&\leq\left| \int_{\R}\delta^{-1}V(\delta^{-1}y)|\psi(y)|^2\dd y-\langle V\rangle|\psi(0)|^2\right|\\
&\leq\left| \int_{\R}V(y)|\psi(\delta y)|^2\dd y-\langle V\rangle|\psi(0)|^2\right|\\
&\leq \int_{\R}|V(y)|\left||\psi(\delta y)|^2-|\psi(0)|^2\right|\dd y\,.
\end{split}\]
It is well-known that
\[\left||\psi(\delta y)|^2-|\psi(0)|^2\right|\leq \delta^{\frac12} \|(|\psi|^2)'\|\sqrt{|y|}\leq 2\delta^{\frac12} \|\psi\psi'\|\sqrt{|y|} \,,\]
and also that, by the usual Sobolev embedding,
\[\|\psi\|_{L^{\infty}(\R)}\leq C\|\psi\|_{H^1(\R)}\,.\]
Thus,
\[\left||\psi(\delta y)|^2-|\psi(0)|^2\right|\leq  C\delta^{\frac12} \|\psi\|^2_{H^1(\R)}\sqrt{|y|} \,,\]
so that
\[|p_\delta(\psi)-p^{\mathrm{eff}}(\psi)|\leq C\delta^{\frac12}\|\psi\|^2_{H^1(\R)}\,.\]
Therefore,
\[p_\delta^-(\psi)\leq p_\delta(\psi)\leq p_\delta^+(\psi)\,,\]
where
\[\begin{split} p_\delta^-(\psi)&=(1-C\delta^{\frac12})\|\psi'\|^2+\langle V\rangle|\psi(0)|^2-C\delta^{\frac12}\|\psi\|^2\,,\\
p_\delta^+(\psi)&=(1+C\delta^{\frac12})\|\psi'\|^2+\langle V\rangle|\psi(0)|^2+C\delta^{\frac12}\|\psi\|^2\,.\end{split}\]
If $\nu(\delta)$ is the bottom of the spectrum of $M_\delta$, we have
\[\nu(\delta)=-\frac{\langle V\rangle^2}{4}+\mathcal{O}(\delta^{\frac12})<0\,.\]

\bibliographystyle{abbrv}
\bibliography{biblio}

\begin{thebibliography}{10}

\bibitem{BN}
V.~Bonnaillie.
\newblock On the fundamental state energy for a {S}chr\"{o}dinger operator with
  magnetic field in domains with corners.
\newblock {\em Asymptot. Anal.}, 41(3-4):215--258, 2005.

\bibitem{BHR22}
V.~Bonnaillie-No\"{e}l, F.~H\'{e}rau, and N.~Raymond.
\newblock Purely magnetic tunneling effect in two dimensions.
\newblock {\em Invent. Math.}, 227(2):745--793, 2022.

\bibitem{BMR14}
V.~Bruneau, P.~Miranda, and G.~Raikov.
\newblock Dirichlet and {N}eumann eigenvalues for half-plane magnetic
  {H}amiltonians.
\newblock {\em Rev. Math. Phys.}, 26(2):1450003, 23, 2014.

\bibitem{CG}
M.~Correggi and E.~L. Giacomelli.
\newblock Almost flat angles in surface superconductivity.
\newblock {\em Nonlinearity}, 34(11):7633--7661, 2021.

\bibitem{Eetal}
P.~Exner, V.~Lotoreichik, and A.~P\'{e}rez-Obiol.
\newblock On the bound states of magnetic {L}aplacians on wedges.
\newblock {\em Rep. Math. Phys.}, 82(2):161--185, 2018.

\bibitem{FH06}
S.~Fournais and B.~Helffer.
\newblock Accurate eigenvalue asymptotics for the magnetic {N}eumann
  {L}aplacian.
\newblock {\em Ann. Inst. Fourier (Grenoble)}, 56(1):1--67, 2006.

\bibitem{FH10}
S.~Fournais and B.~Helffer.
\newblock {\em Spectral methods in surface superconductivity}, volume~77 of
  {\em Progress in Nonlinear Differential Equations and their Applications}.
\newblock Birkh\"{a}user Boston, Inc., Boston, MA, 2010.

\bibitem{HK17}
B.~Helffer and A.~Kachmar.
\newblock Eigenvalues for the {R}obin {L}aplacian in domains with variable
  curvature.
\newblock {\em Trans. Amer. Math. Soc.}, 369(5):3253--3287, 2017.

\bibitem{Raymond}
N.~Raymond.
\newblock {\em Bound states of the magnetic {S}chr\"{o}dinger operator},
  volume~27 of {\em EMS Tracts in Mathematics}.
\newblock European Mathematical Society (EMS), Z\"{u}rich, 2017.

\bibitem{Simon}
B.~Simon.
\newblock The bound state of weakly coupled {S}chr\"{o}dinger operators in one
  and two dimensions.
\newblock {\em Ann. Physics}, 97(2):279--288, 1976.

\end{thebibliography}

\end{document}